\documentclass[draft]{amsart}
\usepackage{hyperref}
\hypersetup{nesting=true,debug=true,naturalnames=true}
\usepackage{graphicx,amssymb}

\usepackage[square,sort&compress,comma,numbers]{natbib}
\usepackage{nicefrac,xcolor,upref,version}

\usepackage{amscd,amsthm,amsmath,amssymb,amsfonts}

\usepackage{mathrsfs}

\newtheorem{theorem}{Theorem}[section]
\newtheorem{lemma}[theorem]{Lemma}

\numberwithin{equation}{section}

\def\Q{{\mathbb {Q}}}

\def\Z{{\mathbb Z}} 
 
\def\C{{\bf C}}  
 
   \def\resp{{\rm resp.  }}

\def\house#1{\setbox1=\hbox{$\,#1\,$}%
\dimen1=\ht1 \advance\dimen1 by 2pt \dimen2=\dp1 \advance\dimen2 by 2pt
\setbox1=\hbox{\vrule height\dimen1 depth\dimen2\box1\vrule}%
\setbox1=\vbox{\hrule\box1}%
\advance\dimen1 by .4pt \ht1=\dimen1
\advance\dimen2 by .4pt \dp1=\dimen2 \box1\relax}

\def\Norm{{\rm Norm}}

  \def\eps{{\varepsilon}}

\def\build#1_#2^#3{\mathrel{\mathop{\kern 0pt#1}\limits_{#2}^{#3}}}

\def\date {le\ {\the\day}\ \ifcase\month\or
janvier\or fevrier\or mars\or avril\or mai\or juin\or juillet\or
ao\^ut\or septembre\or octobre\or novembre\or
d\'ecembre\fi\ {\oldstyle\the\year}}

\font\fivegoth=eufm5 \font\sevengoth=eufm7 \font\tengoth=eufm10

\newfam\gothfam \scriptscriptfont\gothfam=\fivegoth
\textfont\gothfam=\tengoth \scriptfont\gothfam=\sevengoth

\def \C{\mathbb{C}}

\def \Q{\mathbb{Q}}

\def \Z{\mathbb{Z}}

\def\cN{{\mathcal{N}}}
\def \Norm{{\rm Norm}}

\def\mueff{{\mu_{\rm eff}}} 
\def\nueff{{\nu_{\rm eff}}}

\def\den{{\rm den}}

\begin{document}

\title{Fractional parts of powers of real algebraic numbers}
\author{ Yann Bugeaud} 
\address{Universit\'e de Strasbourg, Math\'ematiques,
7, rue Ren\'e Descartes, 67084 Strasbourg  (France)}
\email{bugeaud@math.unistra.fr}

\bigskip
\begin{abstract} 
Let $\alpha$ be a real number greater than $1$. 
We establish an effective lower bound for the distance between 
an integral power of $\alpha$ and its nearest integer. 
\end{abstract}
\maketitle

\section{introduction}

For a real number $x$, let 
$$
||x|| = \mbox{min}\{|x-m|:m\in\mathbb{Z}\}
$$
denote its distance to the nearest integer. 
Not much is known on the distribution of the sequence $(\| \alpha^n \|)_{n \ge 1}$ 
for a given real number $\alpha$ greater than $1$. 
For example, we do not know whether 
the sequence $(\| (3/2)^n \|)_{n \ge 1}$ is dense in $[0, 1/2]$, nor whether $\| {\rm e}^n \|$ tends to $0$
 as $n$ tends to infinity. In 1957 Mahler \cite{Mah57} applied Ridout's $p$-adic extension 
 of Roth's theorem to prove the following result. 

\begin{theorem}\label{Mah}
Let $r/s$ be a rational number greater than $1$ and which is not an integer. 
Let $\eps$ be a positive real number. 
Then, there exists an integer $n_0$ such that 
$$
\| (r/s)^n \| > s^{- \eps n}, 
$$
for every integer $n$ exceeding $n_0$. 
\end{theorem}

In a breakthrough paper, Corvaja and Zannier \cite{CZ04} 
applied ingeniously the $p$-adic Schmidt Subspace Theorem to 
extend Theorem \ref{Mah} to real algebraic numbers. 
Recall that a Pisot number is a real algebraic integer greater than $1$ with the property 
that all of its Galois conjugates (except itself) lie in the open unit disc.

\begin{theorem}\label{CZ}
Let $\alpha$ be a real algebraic number greater than $1$ and $\eps$ a positive real number. 
If there are no positive integers $h$ such that the real number $\alpha^h$ is a Pisot number, then 
there exists an integer $n_0$ such that 
$$
\| \alpha^n \| > \alpha^{- \eps n}, 
$$
for every integer $n$ exceeding $n_0$. 
\end{theorem}

Let $\alpha>1$ be a real algebraic number and $h$ a positive integer 
such that $\alpha^h$ is a Pisot number of degree $d$.  
Then there exists a positive real number $\eta$ such that the modulus of any Galois conjugate (except itself)
of $\alpha^h$ is no greater than $\alpha^{- \eta}$. Let $n$ be a positive integer. 
Since the trace of $\alpha^{h n}$ is a rational integer, we get 
$\| \alpha^{h n} \| \le d \alpha^{- \eta n}$. 
This shows that the restriction in Theorem \ref{CZ} is necessary.

Theorems \ref{Mah} and \ref{CZ} are ineffective in the sense that their proof does not yield an 
explicit value for the integer $n_0$. 
To get an effective improvement on the trivial estimate $\| (r/s)^n \| \ge s^{-n}$, 
Baker and Coates \cite{BaCo75} (see also \cite{Bu02} and \cite[Section 6.2]{Bu18})
used the theory of linear forms in $p$-adic logarithms, for a prime number $p$ dividing~$s$.

\begin{theorem}\label{BC}
Let $r/s$ be a rational number greater than $1$ and which is not an integer. 
Then, there exist an effectively computable positive real number $\tau$ 
and an effectively computable integer $n_0$ such that 
$$
\| (r/s)^n \| > s^{- (1 - \tau) n}, 
$$
for every integer $n$ exceeding $n_0$. 
\end{theorem}

The purpose of this note is to extend Theorem \ref{BC} to 
real algebraic numbers exceeding $1$. 
At first, we have to see which estimate   
follows from a Liouville-type inequality, which boild down to say that 
any nonzero rational integer has absolute value at least $1$. 
To simplify the discussion, take $\alpha$ a real algebraic 
integer greater than $1$ and of degree $d \ge 2$ such that 
each of its Galois conjugates $\alpha_2, \ldots , \alpha_d$ 
has absolute value at most equal to $\alpha$. 
For a positive integer $n$, let $A_n$ be the integer such that 
$$
\| \alpha^n \| = | \alpha^n - A_n |.   
$$
Observe that every Galois conjugate of $\alpha^n - A_n$ has modulus less than $3 \alpha^n$. 
Noticing that the 
absolute value of the norm of the nonzero algebraic integer $\alpha^n - A_n$ 
is at least equal to $1$, we deduce that 
$$
\| \alpha^n \| \ge 3^{-(d-1)} \, \alpha^{-n (d-1)}.     \eqno (1.1)
$$
This is much weaker than what follows from Theorem \ref{CZ}, but this is effective. 
For an arbitrary real algebraic number greater than $1$, a similar argument gives the 
following statement. In the sequel, an empty product is understood to be equal to~$1$. 


\begin{lemma}\label{Liouv}
Let $\alpha$ be a real algebraic 
number greater than $1$ and of degree $d \ge 1$. 
Let $a_d$ denote the leading 
coefficient of its minimal defining polynomial over $\Z$ and 
$\alpha_1, \ldots , \alpha_d$ its Galois conjugates, ordered in such a way 
that $|\alpha_1| \le \ldots \le |\alpha_d|$. Let $j$ be such that $\alpha = \alpha_j$. 
Set 
$$
C(\alpha) = a_d  \, \alpha^{d-1} \, \prod_{i > j} {|\alpha_i| \over  \alpha }.     
$$
If $\alpha$ is not an integer root of an integer, then we have
$$
\| \alpha^n\| \ge 3^{-(d-1)} \, C(\alpha)^{-n}, \quad \hbox{for $n \ge 1$}.   \eqno (1.2)
$$
Otherwise, (1.2) holds only for the positive integers $n$ such that 
$\alpha^n$ is not an integer. 
\end{lemma}

We will see how the theory of linear forms in logarithms allows us to slightly 
improve (1.2), unless there exists a positive integer $h$ such that $\alpha^h$ is an 
integer or a quadratic Pisot unit. 
In the latter case, $\alpha^h$ is a root of an integer polynomial of the shape $X^2 - a X + b$, 
with $a \ge 1$, $b \in \{-1, 1\}$, and $(a, b) \notin \{(1, 1), (2, 1)\}$, thus 
$\alpha = (a + \sqrt{a^2 - 4 b})/2$ and 
$\| \alpha^{hn} \| = \alpha^{-hn}$. 
Except in these cases, we establish the following effective strengthening of Lemma \ref{Liouv}.


\begin{theorem}\label{Main} 
Let $\alpha$ be a real algebraic 
number greater than $1$. 
Let $C(\alpha)$ be as in the statement of Proposition \ref{Liouv}. 
Let $h$ be the smallest positive integer such that $\alpha^h$ is an integer or a quadratic Pisot unit
and put $\cN_\alpha = \{ h n : n \in \Z_{\ge 1} \}$. 
If no such integer exists, then put $\cN_\alpha = \Z_{\ge 1}$. 
There exist a positive, effectively computable real number $\tau = \tau(\alpha)$ 
and an effectively computable integer $n_0 = n_0(\alpha)$, both 
depending only on $\alpha$,  such that 
$$
\| \alpha^n\| \ge C(\alpha)^{- (1 - \tau) n}, \quad \hbox{for $n > n_0$ in $\cN_\alpha$.} 
$$
\end{theorem}


Theorem \ref{Main} should be compared with the effective improvement of Liouville's upper bound for 
the irrationality exponent of an irrational, algebraic real number. 
Recall that the irrationality exponent $\mu (\xi)$ of an irrational real number $\xi$ is given by
$$
\mu (\xi) = 1 + \limsup_{q \to + \infty} \, {- \log \| q \xi \| \over \log q}. 
$$
Its effective irrationality exponent $\mueff (\xi)$ is the infimum of the real numbers $\mu$ 
for which there exists an effectively computable positive integer $q_0$ such that the upper bound
$1 + (- \log \| q \xi \| )/ (\log q) \le \mu$ holds for $q \ge q_0$. 
Let $\alpha$ be an algebraic real number of degree $d \ge 2$. 
Roth's theorem asserts that $\mu (\alpha) = 2$, while Liouville's theorem 
says that $\mueff (\alpha) \le d$. 
By means of the theory of linear forms in logarithms, 
Feldman \cite{Fe68} proved the existence of an effectively computable positive 
real number~$\tau' = \tau'(\alpha)$, 
depending on $\alpha$, such that $\mueff (\alpha) \le (1 - \tau')d$. 

Here, the situation is similar. 
For a real number $\xi$ not an integer, nor a root of an integer, define 
$$
\nu (\xi) = \limsup_{n \to + \infty} \, {- \log \| \xi^n \| \over n}
$$
and let $\nueff (\xi)$ denote the infimum of the real numbers $\nu$ 
for which there exists an effectively computable integer $n_0$ such that 
$(- \log \| \xi^n \| )/n \le \nu$ for $n \ge n_0$. 

Let $\alpha > 1$ be an algebraic real number. 
Theorem \ref{CZ} asserts that $\nu (\alpha) = 0$, unless $\alpha$ is an integer root of a Pisot number. 
Lemma \ref{Liouv} says that $\nueff (\alpha) \le \log C(\alpha)$, while Theorem \ref{Main} 
slightly improves the latter inequality. Furthermore, the positive real number $\tau (\alpha)$ 
occurring in Theorem \ref{Main} is very small and of comparable size as 
the real number $\tau' (\alpha)$, when $\alpha$ is an algebraic integer (otherwise, it also 
depends on the prime factors of the leading coefficient of the minimal defining polynomial of $\alpha$
over $\Z$).

Among the many open questions on the function $\nu$, let us mention that we do not know 
whether $\nu ({\rm e})$ is finite or not (see \cite[Problem 13.20]{Bu18}
and \cite{BuDu08} for further results and questions). 
Mahler and Szekeres \cite{MaSz67} established that, with respect to the 
Lebesgue measure, almost all real numbers $\xi$ satisfy 
$\nu (\xi) = 0$. Furthermore, the set of real numbers $\xi$ 
such that $\nu (\xi)$ is infinite has Hausdorff dmension zero \cite[Theorem 3]{BuDu08}.

Sometimes, the hypergeometric method yields better improvements of (1.2). 
This is the case for the algebraic numbers $\sqrt{2}$ and $3/2$, see 
Beuker's seminal papers \cite{Beu80,Beu81} 
and the subsequent works \cite{BaBe02,Zu07} where it is shown that 
$$
\nueff ( \sqrt{2} ) \le 0.595, \quad \nueff ( 3/2) < 0.5443, 
$$
respectively. 

\goodbreak

\section{Proofs}

\begin{proof}[Proof of Lemma \ref{Liouv}] 
We keep the notation of the lemma and follow the proof of \cite[Assertion (a)]{MaSz67} 
with a slight improvement. 

Let $n$ be a positive integer. Observe that the polynomial 
$$
f_n(X) = a_d^n (X - \alpha_1^n) \cdots (X - \alpha_d^n)
$$
has integer coefficients and denote by $A_n$ the integer such that 
$$
\|   \alpha^n \| = | \alpha^n - A_n |.
$$
If $\alpha^n$ is not an integer, then 
$f(A_n)$ is a nonzero integer and we get
$$
|f(A_n)| \ge 1,  \eqno (2.1) 
$$
thus, 
$$
| \alpha^n - A_n | \ge a_d^{-n} \, \prod_{1 \le i \le d, i \not= j} \, |\alpha_i^n - A_n|^{-1}.  
$$
For $i = 1, \ldots , d$, note that 
$$
|\alpha_i^n - A_n| \le |\alpha_i|^n + \alpha^n + 1 \le 3 (\max\{|\alpha_i|, \alpha\})^n. 
$$
Consequently, we obtain the lower bound
$$
\| \alpha^n\| \ge 3^{-(d-1)} \, a_d^{-n} \, \alpha^{-(d-1)n} \, \prod_{i > j} {\alpha^n \over  |\alpha_i|^n},
$$
as claimed. 
This inequality reduces to (1.1) if $a_d = 1$ and $j = d$. 
\end{proof}

The proof of Theorem \ref{Main} makes use of the following result of 
Boyd \cite{Boyd94}.

\begin{lemma}\label{Boyd}
Let $f(X)$ be an irreducible polynomial of degree $d$ with integer coefficients. 
Let $m$ denote the number of roots of $f(X)$ of maximal modulus. 
Assume that one of these roots is real and positive. 
Then $m$ divides $d$ and there is an irreducible polynomial $g(X)$ with integer 
coefficients such that $f(X) = g(X^m)$.
\end{lemma}


\begin{proof}[Proof of Theorem \ref{Main}] 
We proceed in a similar way as when dealing with Thue equations. 
In view of Theorem \ref{BC} we assume that $\alpha$ is irrational. 
Let $K$ denote the number field $\Q (\alpha)$. 
Let $h$ denote the absolute Weil height. 
For convenience, we define the function $h^* ( \cdot ) = \max\{h (\cdot) , 1\}$. 
The constants $c_1, c_2, \ldots $ below are positive, effectively computable, 
and depend only on $\alpha$. 

Let $a_d$ denote the leading 
coefficient of the minimal defining polynomial of $\alpha$ over $\Z$
and $S$ the set of places of $K$ composed of all the infinite places and all the places 
corresponding to a prime ideal dividing $a_d$. 
Let $N_S$ denote the $S$-norm. 
We direct the reader to \cite[Chapter 1]{EvGy15} for definitions and basic results. 
Let us only mention that if the absolute value of the norm of a nonzero element $\beta$ in $K$ 
is written as $|\Norm_{K / \Q} (\beta)| = a_S b$, where every prime divisor 
of $a_S$ divides $a_d$ and no prime divisor of $b$ divides $a_d$, then $N_S (\beta) = b$. 
In particular, if $a_d = 1$, then $N_S$ is the absolute value of the norm $\Norm_{K / \Q}$. 

Let $n$ be a positive integer and $A_n$ denote the integer such that 
$$
\|  \alpha^n \| = | \alpha^n - A_n |.
$$
Put $\delta_n = \alpha^n - A_n$. 
We will obtain a lower bound of the form $\kappa^n$ with $\kappa > 1$ 
for the $S$-norm of the nonzero $S$-integer $\delta_n$. 
By replacing in the proof of Lemma \ref{Liouv} the right hand side of (2.1) by $\kappa^n$, we
obtain the expected improvement. 

Let $\eta_1, \ldots , \eta_s$ be a fundamental system of $S$-units in $K$. 
By \cite[Proposition 4.3.12]{EvGy15}, 
there exist integers $b_1, \ldots , b_s$ such that 
$$
h \bigl( \delta_n \eta_1^{-b_1} \cdots \eta_s^{-b_s} \bigr) 
\le {\log N_S (\delta_n) \over d} + c_1.      \eqno (2.2)
$$
Since 
$$
h(\delta_n) \le n h(\alpha) + \log A_n + \log 2 \le n h(\alpha) + n \log \alpha + 2 \log 2,
$$
it follows from \cite[Proposition 4.3.9 (iii)]{EvGy15} and (2.2) that 
$$
B := \max\{|b_1|, \ldots , |b_s| \} \le c_2  h^* (\delta_n) \le c_3 n. 
$$
Set $\gamma_n = \delta_n \eta_1^{-b_1} \cdots \eta_r^{-b_r}$.

Assume first that there exists a Galois conjugate $\beta$ of 
$\alpha$ such that $|\beta| > \alpha$ and consider the quantity
$$
\Lambda_n = { \beta^n - A_n \over   \beta^{n} }. 
$$
Observe that 
$$
0 < |\Lambda_n - 1| \le 2^{-c_4 n}.
$$
Let $\sigma$ denote the embedding sending $\alpha$ to $\beta$ and observe that 
$$
\Lambda_n  = \sigma (\gamma_n)  \beta^{-n} \sigma (\eta_1)^{b_1} \cdots \sigma (\eta_r)^{b_r}.
$$
We apply the theory of linear forms in logarithms: it follows from \cite[Theorem 2.1]{Bu18} that
$$
\log |\Lambda_n - 1| \ge - c_5 h^* (\gamma_n) \log \Bigl( {B + n \over h^* (\gamma_n)} \Bigr),
$$
giving
$$
n \le c_6 h^* (\gamma_n) \log \Bigl( {n \over h^* (\gamma_n)} \Bigr).
$$
We derive that 
$$
n \le c_7 h^* (\gamma_n) \le c_8 \log N_S (\delta_n)  + c_9, 
$$
thus
$$
N_S (\delta_n) \ge 2^{c_{10} n}, \quad \hbox{for $n \ge c_{11}$}. 
$$
This improves the trivial lower bound $N_S (\delta_n) \ge 1$ used 
in the proof of Lemma \ref{Liouv}. 

Secondly, we assume that the modulus of every Galois conjugate of $\alpha$ 
is less than or equal to $\alpha$. 
By Lemma \ref{Boyd}, there exist a divisor $m$ of $d$ and an irreducible integer polynomial $g(X)$
of degree $d/m$
such that $f(X)$ has exactly $m$ roots of modulus $\alpha$ 
and the minimal defining polynomial $f(X)$ over $\Z$ satisfies $f(X) = g(X^m)$. 

Assume that $d / m \ge 2$.
If $f(X)$ has a root $\beta$ of modulus at least equal to $1$ and different from $\alpha$, 
then $A_n - \alpha^n$ cannot be equal to $\beta^n$, thus the quantity 
$$
\Lambda'_n = {A_n -  \beta^n \over \alpha^{n} }   \eqno (2.3) 
$$
satisfies
$$
0 < |\Lambda'_n - 1| \le 2^{-c_{12} n}.   \eqno (2.4)
$$
We get a lower bound for $|\Lambda'_n - 1|$ by proceeding exactly as above, and it takes the 
same shape as our lower bound for $ |\Lambda_n - 1|$. 
We then deduce the lower bound 
$$
|N_S (\delta_n)| \ge 2^{c_{13} n}, \quad \hbox{for $n \ge c_{14}$}. 
$$

Now, we assume that all the roots of $f(X)$, except $\alpha$, lie in the open unit disc. 

If $\alpha$ has two real Galois conjugates in the open unit disc, then one of them, denoted by $\beta$, 
is such that the quantity $\Lambda'_n$ defined as in (2.3) is not equal to $1$
and (2.4) holds. We argue as above to get a similar lower bound 
for $N_S (\delta_n)$. 

If $d/m \ge 3$ and $\alpha^m$ has a complex nonreal Galois conjugate $\beta^m$ in the 
open unit disc, then $\beta^j$ is complex nonreal for every positive integer $j$ 
and we proceed as above, since the quantity $\Lambda'_n$ defined as in (2.3) is not equal to $1$. 

Consequently, we can assume that $d / m = 2$ and $g(X)$ is the minimal defining polynomial over $\Z$ 
of the quadratic number $\alpha^m$. 

If $n$ is not a multiple of $m$, then there exists a Galois conjugate $\beta$ of $\alpha$
such that $\beta^n$ is complex nonreal, thus the quantity 
$\Lambda'_n$ defined above is not equal to $1$, and we can proceed exactly as above 
to get a similar lower bound for $N_S (\delta_n)$. 

Assume now that $n$ is a multiple of $m$. 
Write $g(X) = a_2 X^2 - u X - v$. 
Denote by $\sigma(\alpha)$ the Galois conjugate of $\alpha$. 
If $\alpha$ is not an algebraic integer, then there exists a prime number $p$ such that $v_p (\alpha) < 0$. 
Since $v_p (\alpha) \le -1/2$, it follows from
\cite[Theorem B.11]{Bu18} that the $p$-adic valuation of 
$\alpha^n + \sigma(\alpha)^n$ satisfies
$$
v_p \bigl(  \alpha^n + \sigma(\alpha)^n \bigr) \le n v_p (\alpha) + c_{15} \log n \le - {n \over 3},
$$
for $n \ge c_{16}$. In particular, for $n$ greater than $c_{16}$, the algebraic number 
$\alpha^n +  \sigma(\alpha^n)$ cannot be a rational integer. 
Then, the quantity 
$\Lambda'_n$ defined above is not equal to $1$, and we can proceed exactly as above 
to get a similar lower bound for $N_S (\delta_n)$. 

If $\alpha$ is an algebraic integer, then $a_2 = 1$ and
$\alpha^n +  \sigma(\alpha^n)$ is equal to the nearest integer $A_n$
to $\alpha^n$. Thus, we have
$$
\| \alpha^n \| = | \sigma(\alpha^n)| = { |v|^{n/m} \over  \alpha^n},
$$
while Lemma \ref{Liouv} asserts that 
$$
\| \alpha^n \|  \ge 3^{-1}  \alpha^{-n}.   
$$
Consequently, we obtain the desired improvement on (1.2) if $|v| \ge 2$. 
As already noticed, (1.2) is essentially best possible if $|v| = 1$.

It only remains for us to consider the case $d=m$.
Then, there exist coprime nonzero integers $u$, $v$ 
with $u > v > 0$ 
such that the minimal defining polynomial of $\alpha$ over $\Z$ is $v X^d - u$. 
If $v = 1$, then $\alpha$ is the $d$-th root of the integer $u$. 
If $d=2$, then 
$$
\| \sqrt{u} u^m \| \ge u^{-( 1 - c_{17})m},  \quad \hbox{for $m \ge 1$}, 
$$
by \cite[Theorem 1.2]{BeBu12} (see also \cite[Theorem 6.3]{Bu18}). 
If $d \ge 3$ and $j = 1, \ldots , d-1$, then it follows from an effective improvement of Liouville's bound
$d$ for the irrationality exponent of $u^{j/d}$ (see \cite[Section 6.3]{Bu18}) that 
$$
\| u^{j/d} u^m \| \ge u^{-( 1 - c_{18})(d-1) m}, \quad \hbox{for $m \ge 1$}. 
$$
In both cases, noticing that $C(\root{d}\of{u}) = u^{(d-1)/d}$, we get 
$$
\| (\root{d}\of{u})^n \| \ge C(\root{d}\of{u})^{-( 1 - c_{19}) n}, 
\quad \hbox{for $n \ge 1$ not a multiple of $d$},
$$
as expected. 
Now, assume that $v \ge 2$. We argue in a similar way as in the proof of Theorem \ref{BC}. 
Let $p$ be a prime divisor of $v$. 
Write 
$$
{\delta_n \over v^{n/d} } = \biggl( \root{d}\of{ {u \over v}} \biggr)^n - A_n
$$
and note that the $p$-adic valuation of 
$$
\Omega_n = v^{n/d} A_n =   ( u^{1/d} )^n - \delta_n       
$$
satisfies $v_p (\Omega_n) \ge c_{20} n$. 
Let $L$ denote the number field generated by $u^{1/d}$ and $v^{1/d}$. Let $S$ be the 
set of places of $L$ composed of all the infinite places and all the places corresponding to 
a prime ideal dividing $v$. 
Let $\eta_1, \ldots , \eta_s$ be a fundamental system of units in $L$. 
By \cite[Proposition 4.3.12]{EvGy15}, 
there exist integers $b_1, \ldots , b_r$ such that 
$$
h \bigl( \delta_n \eta_1^{-b_1} \cdots \eta_r^{-b_r} \bigr) 
\le { \log N_S (\delta_n)  \over d}  + c_{21}.      \eqno (2.5)
$$
Since $h(\delta_n) \le c_{22} n$,
it follows from \cite[Proposition 4.3.9 (iii)]{EvGy15} and (2.5) that 
$$
B := \max\{|b_1|, \ldots , |b_r| \} \le c_{23}  h^* (\delta_n) \le c_{24} n. 
$$
Set $\gamma_n = \delta_n \eta_1^{-b_1} \cdots \eta_r^{-b_r}$ and note that 
$$
\Omega_n = ( u^{1/d} )^n - \gamma_n \eta_1^{b_1} \cdots \eta_r^{b_r}. 
$$
It follows from the theory of linear forms 
in $p$-adic logarithms, more precisely, from \cite[Theorem 2.11]{Bu18}, that 
$$
v_p (\Omega_n) \le c_{25} h^* (\gamma_n) \log \Bigl( {B + n \over h^* (\gamma_n)} \Bigr). 
$$
This gives
$$
n \le c_{26} h^* (\gamma_n) \log \Bigl( {n \over h^* (\gamma_n)} \Bigr),
$$
and we derive that 
$$
n \le c_{27} h^* (\gamma_n) \le c_{28} \log N_S (\delta_n)  + c_{29}, 
$$
thus
$$
N_S (\delta_n) \ge 2^{c_{30} n}, \quad \hbox{for $n \ge c_{31}$}. 
$$
This improves the trivial lower bound $N_S (\delta_n) \ge 1$ used 
in the proof of Lemma \ref{Liouv}. 
This concludes the proof of the theorem. 
\end{proof}

\end{document}